\documentclass{amsart}

\usepackage{lmodern}
\usepackage[T1]{fontenc}
\usepackage[utf8]{inputenc}

\usepackage{overpic}

\usepackage{amsmath,amssymb,amsfonts,amsthm}
\usepackage[nobysame]{amsrefs}
\usepackage{mathtools}
\usepackage{xspace}

\newtheorem{theorem}{Theorem}
\newtheorem{proposition}{Proposition}
\newtheorem{corollary}{Corollary}
\newtheorem{lemma}{Lemma}
\theoremstyle{definition}

\theoremstyle{remark}
\newtheorem{example}{Example}
\newtheorem{remark}{Remark}

\usepackage{hyperref}
\hypersetup{
  pdfauthor={Zijia Li, Hans-Peter Schröcker, Johannes Siegele},
  pdftitle={A Geometric Algorithm for the Factorization of Spinor Polynomials},
  pdfkeywords={Conformal motion, elementary motion, null displacement, Study variety, null quadric, four quaternion representation},
  colorlinks=true,
  allcolors=blue,
}

\renewcommand{\H}{\mathbb{H}}
\renewcommand{\P}{\mathbb{P}}
\newcommand{\R}{\mathbb{R}}
\newcommand{\C}{\mathbb{C}}
\renewcommand{\DH}{\mathbb{DH}}
\newcommand{\SV}{\mathcal{S}}
\newcommand{\NQ}{\mathcal{N}}
\newcommand{\SC}{S}
\renewcommand{\SS}{\mathbb{S}}
\newcommand{\CGA}{\ensuremath{\operatorname{CGA}}\xspace}
\newcommand{\CGAp}{\ensuremath{\CGA_+}\xspace}
\newcommand{\SO}[1][3]{\ensuremath{\operatorname{SO}(#1)}\xspace}
\newcommand{\SE}[1][3]{\ensuremath{\operatorname{SE}(#1)}\xspace}
\newcommand{\reverse}[1]{\widetilde{#1}}
\newcommand{\ci}{\mathrm{i}}
\newcommand{\qi}{\mathbf{i}}
\newcommand{\qj}{\mathbf{j}}
\newcommand{\qk}{\mathbf{k}}
\newcommand{\eps}{\varepsilon}
\newcommand{\norm}[1]{\Vert #1\Vert}
\newcommand{\Cl}[1]{\mathcal{C}\kern-0.4mm\ell(#1)}

\DeclareMathOperator{\Vect}{v}

\title[Factorization of Spinor Polynomials]{A Geometric Algorithm for the Factorization of Spinor Polynomials}

\date{\today}

\author{Zijia Li}
\address{KLMM, Academy of Mathematics and Systems Science,
  Chinese Academy of Sciences, Beijing, China}
\email{lizijia@amss.ac.cn}

\author{Hans-Peter Schröcker}
\address{Department of Basic Sciences in Engineering, University of Innsbruck,
  Innsbruck, Austria}
\email{hans-peter.schroecker@uibk.ac.at}

\author{Johannes Siegele}
\address{Department of Basic Sciences in Engineering, University of Innsbruck,
  Innsbruck, Austria}
\email{johannes.siegele@uibk.ac.at}


\begin{document}

\begin{abstract}
  We present a new algorithm to decompose generic spinor polynomials into linear
  factors. Spinor polynomials are certain polynomials with coefficients in the
  geometric algebra of dimension three that parametrize rational conformal
  motions. The factorization algorithm is based on the ``kinematics at
  infinity'' of the underlying rational motion. Factorizations exist
  generically but not generally and are typically not unique. We prove that
  generic multiples of non-factorizable spinor polynomials admit factorizations
  and we demonstrate at hand of an example how our ideas can be used to tackle
  the hitherto unsolved problem of ``factorizing'' algebraic motions.
\end{abstract}

\maketitle

\section{Introduction}
\label{sec:introduction}

The article \cite{hegedus13:_factorization2} presented an algorithm to decompose
a rational rigid body motion into a sequence of rotations or translations,
coupled by the same rational motion parameter. In the Study model \cite{study} of space kinematics  this decomposition corresponds to the factorization of special
polynomials over the ring of dual quaternions (``motion polynomials'') into
linear factors. The algorithm to do so builds upon older ideas for factorizing 
quaternion polynomials \cites{niven41,gordon65}. In contrast to the quaternion
case, the factorization of motion polynomials has to fail in some non-generic
cases because not all motion polynomials admit factorizations with linear
factors.

The rotations or translations parametrized by linear factors can be realized by
mechanical joints. Since factorizations of motion polynomials are generically non-unique, it is possible to combine different factorizations into closed-loop
linkages. The thus resulting movable mechanical structures are of importance in theoretical mechanism science \cites{gallet16,li18:_universality_theorem} and
have some potential for engineering applications
\cites{hegedus15:_four_pose_synthesis,liu21,liu23}. Note that our notion of non-uniqueness is weaker than what is commonly found in some algebra communities, as discussed in the survey papers \cites{geroldinger2006non,baeth2015factorization,smertnig2016factorizations}.

As has been observed at several occasions, the algebraic factorization algorithm
of \cite{hegedus13:_factorization2} generalizes, under certain conditions, to
polynomials with coefficients in other algebras. It is worth to point out that the characterization \cites{gentili2021zeros,gentili2022regular} for zeros of slice functions \cite{gentili2007new}  also sheds some light on the factorization of motion polynomials over dual quaternions. In the articles
\cites{Hestenes2001,Li2001a,Li2001b,Li2001c}, the authors have provided an
in-depth analysis of the principles and applications of the general conformal
geometric algebra (\CGA). An
important example \cite{li19b} is ``projectivized'' spin groups of Clifford algebras \cites{clifford1871,lounesto2001}, which leads to the factorization of what we call ``spinor
polynomials''. The geometric foundations underlying this construction for $\CGA$ of dimension three are the topic of \cite{kalkan22}. In this
article, we present a new factorization algorithm for spinor polynomials in this
algebra that has a pronounced geometric flavor. Since $\CGA$ contains dual
quaternions (and other important algebras) as a sub-algebra, the new algorithm can also be used
to factorize motion polynomials.

We see at least two important advantages of this new algorithm:
\begin{itemize}
\item Firstly, it provides a deep geometric insight into the factorization
  process, allowing us to extend results of
  \cites{li18:_universality_theorem,scharler21} and prove existence of
  factorizable real polynomial multiples for arbitrary spinor polynomials.
\item Secondly, some aspects of this new algorithm are independent of the
  rational parametrization and only depend on the geometric curve. This opens
  possibilities to extend factorization theory to algebraic motions.
\end{itemize}

The latter point touches upon an important generalization of motion polynomial
factorization. In fact, rational motions are rare in mechanism science where
mechanically constrained movable structures typically produce configuration
varieties that are algebraic but not of genus zero. While we believe that this
article is interesting in its own right, it can also be viewed as a important
step towards eastablishing a factorization theory for algebraic motions. In fact, we will even
discuss a basic non-rational example as a proof of concept at the end of this
text, in Section~\ref{sec:four-bar}. Dual quaternions $\DH$, a sub-algebra of
$\CGA$, would be sufficient to describe rigid body kinematics but our
description and derivation will profit a lot from the richer algebraic and
geometric structure of the more general algebra~$\CGA$.

In Section~\ref{sec:preliminaries} we collect necessary concepts and notation.
In Section~\ref{sec:geometric-factorization-algorithm} we describe the new
geometric factorization procedure and prove its correctness. The two following
sections are dedicated to a result on the unconditional factorizability of
suitable real polynomial multiples of spinor polynomials (Theorem~\ref{th:3} in
Section~\ref{sec:factorization-technique}) and to the exemplary factorization of
an algebraic spherical four-bar motion (Section~\ref{sec:four-bar}). Open
questions for a more general factorization theory for algebraic motions will be
discussed at the end of this article. The technical proof of Lemma~\ref{lem:3}
is deferred to an appendix.

\section{Preliminaries}
\label{sec:preliminaries}

In this article, we study polynomials with coefficients in the conformal
geometric algebra (\CGA) in dimension three. We introduce it following the
conventions of \cite{bayro-corrochano19}*{Chapter~8}. Let us take an orthonormal
basis ${e_1,e_2,e_3,e_+,e_-}$ of the quadratic space $\R^{4,1}$ and consider a
multiplication of vectors which satisfies
\begin{equation*}
  e_1^2=e_2^2=e_3^2=e_+^2=-e_-^2=1
\end{equation*}
as well as anti-commutativity of the basis elements:
\begin{equation*}
  e_ie_j = -e_je_i
  \quad\text{for distinct $i$, $j \in \{1,2,3,+,-\}$.}
\end{equation*}
This multiplication extends in a unique way to a real associative algebra of
dimension $32$, the Clifford algebra $\Cl{4,1}$. Conformal geometric algebra
\CGA is obtained by the simple change of basis that replaces $e_-$ and $e_+$
with
\begin{equation*}
  e_o = \frac{1}{2}(e_--e_+)\quad\text{and}\quad e_\infty=e_-+e_+,
\end{equation*}
respectively. Its basis elements are the products of up to five elements from
the set $\{e_1,e_2,e_3,e_\infty,e_o\}$ and we denote them by multiple
subscripts, e.g. we write $e_{ij}$ for $e_ie_j$. The reverse $\reverse{e}_r$ of
a basis element $e_r$ for $r=r_1,r_2,\ldots r_n$ is obtained by simply reversing
the order of multiplication, i.e. $\reverse{e}_r=e_{r_n\ldots r_2,r_1}$. The
grade of $e_r$ is the cardinality of the set $\{r_1,r_2,\ldots,r_n\}$. Dot
product and wedge product of \CGA will be denoted by the symbols ``$\cdot$'' and
``$\wedge$'', respectively.

Any element $q\in\CGA$ can be written as a unique linear combination of the
basis elements. The reverse of $q$ is obtained by reversing each basis element
of this linear combination. Vectors in $\R^{4,1}$ are naturally embedded in
$\CGA$ as grade one elements
\begin{equation}
  \label{eq:1}
  a = a_oe_o+a_1e_1+a_2e_2+a_3e_3+a_\infty e_\infty.
\end{equation}

The main purpose of vectors in \CGA is to represent spheres of three-dimensional
conformal space. The vector \eqref{eq:1} represents the sphere with center
$(a_1,a_2,a_3)/a_o$ and squared radius $a^2$. Negative values of $a^2$ are
possible and lead to spheres with purely imaginary radius. Planes can be viewed
as spheres which contain the point at infinity, i.e. vectors satisfying $a \cdot
e_\infty=0$. Vectors with the property $a\reverse{a}=a^2=0$ represent points in
conformal three-space, i.e. a point $(a_1,a_2,a_3)/a_o\in\R^3$, provided
$a_o\neq 0$, or the point at infinity $e_\infty$ otherwise.

\subsection{Spinors and Conformal Displacements}

If $x$ and $s$ are spheres and $s$ is not a point, the reflection (inversion)
$y$ of $x$ in $s$ is given by the formula
\begin{equation}
  \label{eq:2}
  y = sx\reverse{s}.
\end{equation}
The group of conformal displacement is generated by reflections in spheres or
planes. Denote the even sub-algebra of \CGA by \CGAp. This is the algebra
consisting of linear combinations of basis elements of even grade, such that the
``sandwich product'' \eqref{eq:2} fixes grade one elements. The group of
conformal displacements is isomorphic to the special orthogonal group
$\SO[{4,1}]$ which, in turn, is doubly covered by the spin group
\begin{equation*}
  \{q\in\CGAp\colon q\reverse{q}=\pm 1,\ q a\reverse{q}\in\R^{4,1}\text{ for all $a\in\R^{4,1}$}\}.
\end{equation*}

It is well-known that the spin group of $\Cl{4,1}$ coincides with the group of
even-graded versors, that is, products of an even number of vectors. A spinor
$q$ acts on a vector $a$ via the sandwich product $qa\reverse{q}$ whence $q$ and
$-q$ represent the same displacement. To avoid this representation ambiguity, we
will consider \CGAp modulo the real multiplicative group $\R^\times$. In this
way, elements of $\CGAp/\R^\times$ can be viewed as points of the projective
space $\P(\CGAp) = \P^{15}(\R)$. It provides the scenery for our geometric
factorization algorithm. Whenever we wish to emphasize that elements of
$\CGAp/\R^\times$ should be considered as projective points, we use square
brackets to denote equivalence classes, that is, $[q]=[2q]=[-q]$.

A point $[q] \in \P(\CGAp)$ is represented by a spinor $q$ if
\begin{equation}
  \label{eq:3}
  q\reverse{q}=\reverse{q}q\in\R^\times.
\end{equation}
If $q\reverse{q}=\reverse{q}q$, we will call this the norm of $q$. The group
$\SO[4,1]$ as a point-set is embedded into the projective space $\P(\CGAp)$ as a
projective variety $\SV$ minus a quadric $\NQ$. The variety $\SV$ is defined by
the condition $q\reverse{q}=\reverse{q}q\in\R$ and was called \emph{Study
  variety $\SV$} in \cite{kalkan22}. It generalizes the well-known \emph{Study
  quadric} of rigid body kinematics \cite{selig05}*{Chapter~11}. The \emph{null
  quadric $\NQ$} is given by the condition that the grade zero part of
$q\reverse q$ vanishes. This is a quadratic condition so that $\NQ$ is a quadric
in the classical sense of projective geometry over vector spaces.

It is often necessary for our purposes to also consider the complex extension of
Study variety $\SV$, null quadric $\NQ$, and their ambient projective space
$\P(\CGAp)$. We will therefore tacitly allow $\CGA$ elements with complex
coefficients. This does not change much of the algebraic properties of $\CGAp$
itself but affects a lot important sub-algebras such as the quaternions $\H$ or
the dual quaternions~$\DH$.

\subsection{The Four Quaternion Representation}

Quaternions $\H$ and dual quaternions $\DH$ are embedded in \CGAp via
\begin{equation*}
  \qi\mapsto -e_{23},\quad \qj\mapsto e_{1,3},\quad \qk\mapsto -e_{12},\quad \eps\mapsto e_{123\infty}.
\end{equation*}
In \cite{kalkan22} we showed that any even-graded element $q$ can be written in
a four-quaternion representation
\begin{equation*}
  q=q_0+\eps_1 q_1+\eps_2q_2+\eps_3q_3,
\end{equation*}
with $q_0$, $q_1$, $q_2$, $q_3\in\H$ and $\eps_1=\eps=e_{123\infty}$,
$\eps_2=e_{1230}$, $\eps_3=e_{\infty 0}+1=e_+e_-$. The elements $\eps_1$,
$\eps_2$ and $\eps_3$ commute with the quaternion units; additional
multiplication rules are given in Table~\ref{tab:eps-multiplication1}. The four
quaternion representation groups the sixteen coordinates of \CGAp into four
quadruples and makes manual computations more tractable. We will use it in
examples and in the technical proof of Lemma~\ref{lem:3} in the appendix.

\begin{table}
  \centering
  \caption{Multiplication table for $\eps_1$, $\eps_2$, $\eps_3$ and quaternion
    units.}
  \label{tab:eps-multiplication1}
  \begin{tabular}{c|ccc}
    & $\eps_1$      & $\eps_2$     & $\eps_3$  \\
    \hline
    $\eps_1$ & $0$           & $\eps_3 - 1$ & $\eps_1$  \\
    $\eps_2$ & $-\eps_3 - 1$ & $0$          & $-\eps_2$ \\
    $\eps_3$ & $-\eps_1$     & $\eps_2$     & $1$
  \end{tabular}\qquad
  \begin{tabular}{c|ccc}
    &$\qi$ &$\qj$ &$\qk$ \\
    \hline
    $\qi$ &$-1$ &$\phantom{-}\qk$ &$-\qj$ \\
    $\qj$ &$-\qk$ &$-1$ &$\phantom{-}\qi$ \\
    $\qk$ &$\phantom{-}\qj$ &$-\qi$ &$-1$ \\
  \end{tabular}
\end{table}

\subsection{Spinor Polynomials}

A central object of interest in this article is polynomials $C = \sum_{i=0}^n
t^iq_i$ in the indeterminate $t$ and with coefficients $q_i$ in the even
sub-algebra \CGAp. For polynomials over a non-commutative ring, there are
different notions for multiplication and evaluation. Since we use polynomials to
parametrize rational curves in the Study variety, it is natural to treat the
indeterminate $t$ as a real (or complex) parameter and postulate that it
commutes with all coefficients of the polynomial. The caveat of this convention
is that evaluation of a polynomial for an element not in the center of $\CGAp$
requires an additional convention. We define the \emph{left evaluation of $C$ at
  $h \in \CGAp$} as $C(h) \coloneqq h^iq_i$. A corresponding \emph{right
  evaluation} $\sum_{i=0}^n q_ih_i$ exists and leads to a symmetric theory. We
will not describe the ``right'' theory explicitly but occasionally hint at minor
adaptations that are needed for it to work. Notation to distinguish between left
and right evaluation will not be required.

The reverse of a polynomial is obtained by reversing all of its coefficients,
i.e. $\reverse{C}=\sum_{i=0}^n t^i\reverse{q}_i$. Left and right norm
polynomials are defined as $C\reverse{C}$ and $\reverse{C}C$, respectively. A
polynomial in $\CGAp[t]$ is called spinor polynomial, if
$C\reverse{C}=\reverse{C}C \in \R[t] \setminus \{0\}$. In this case, we will
call $C\reverse{C}$ the \emph{norm polynomial of~$C$.}

Because of $C\reverse{C} = \reverse{C}C \in \R[t] \setminus \{0\}$, a spinor
polynomial $C$ parametrizes a rational curve in the Study variety $\SV$ that is
not entirely contained in the null quadric $\NQ$. We denote by $[C]$ the set of
curve points over the complex numbers $\C$, that is, the set $\{[C(t)] \mid t
\in \C \cup \{\infty\}\}$ with the usual understanding that $C(\infty)$ equals
the leading coefficient of~$C$.

\section{A Geometric Factorization Algorithm}
\label{sec:geometric-factorization-algorithm}

Factorizations into linear factors of univariate left polynomials with
coefficients from the ring of quaternions, dual quaternions, and split
quaternions polynomials was topic of previous research
\cites{niven41,gordon65,hegedus13:_factorization2,li19,scharler21}. They lead to
the decomposition in diverse transformation groups ($\SO$, $\SE$ and
transformations of the hyperbolic planes) into coupled ``elementary motions''
(rotations and, in case of $\SE$, also translations). As all of these algebras
are contained in \CGAp, it seems natural to study factorization with linear
factors of polynomials in $\CGAp[t]$. They correspond to the decomposition of
conformal motions into elementary motions that have been described in
\cite{dorst16} (conformal scaling, translation, and rotation). As has been
observed in \cite{kalkan22}, a standard factorization algorithm for the
mentioned quaternion algebras also works for spinor polynomials $C \in
\CGAp[t]$, at least generically. It is based on the algebra involution
(reversion), a factorization of the real norm polynomial, and division with
remainder for polynomials over non-commutative rings. We briefly present the
basic steps but omit proofs:

\begin{lemma}[Polynomial Division; {\cite{li19b}*{Theorem~1}}]
  \label{lem:1}
  Let $C$, $P\in\CGAp[t]$ be two polynomials and assume that the leading
  coefficient of $P$ is invertible. Then there exist unique polynomials $Q$ and
  $R\in\CGAp[t]$ such that $C=QP +R$ and $\deg(R)<\deg(P)$. Further, for
  $h\in\CGAp$ with $P(h)=0$, it holds $C(h)=R(h)$.
\end{lemma}

A proof for the first part of this (well known) lemma is
\cite{li19b}*{Theorem~1}. The second statement is not trivial because left
evaluation is no ring homomorphism but extending the proof of
\cite{hegedus13:_factorization2}*{Lemma~1} from dual quaternion polynomials to
polynomials in $\CGAp[t]$ is straightforward.

\begin{remark}
  \label{rem:division}
  We will often use Lemma~\ref{lem:1} for polynomials $P \in \R[t]$. In this
  case, we have $C = QP + R = PQ + R$. Also, the second statement of
  Lemma~\ref{lem:1} becomes rather trivial.
\end{remark}

\begin{lemma}[Zeros and Left Factors]
  \label{lem:2}
  Let $C\in\CGAp[t]$ and $h\in\CGAp$. Then $t-h$ is a left factor of $C$ if and
  only if $C(h)=0$.
\end{lemma}

The ``right'' version of this lemma (in the sense of right factors and right
evaluation) is \cite{li19b}*{Theorem~2}.

Using above results, we immediately obtain the following proposition, which
gives a method for the computation of linear left factors.

\begin{proposition}
  \label{prop:generic-factorization}
  Let $C \in \CGAp[t]$ be a spinor polynomial.
  \begin{itemize}
  \item If $M$ is a quadratic, monic, real factor of $C\reverse{C}$, $R$ is the
    (linear) remainder of polynomial division of $C$ by $M$, and $h$ is a common
    zero of $R$ and $M$, then $t-h$ is a left factor of~$C$.
  \item If $t-h$ is a left factor of $C$, then $R(h) = 0$ where $R$ is the
    (linear) remainder of polynomial division of $C$ by $M =
    (t-h)(t-\reverse{h})$.
  \end{itemize}
\end{proposition}

By Proposition~\ref{prop:generic-factorization}, we can find all linear left
factors of a spinor polynomial $C$ by computing all quadratic, monic, real
factors $M$ of $C\reverse{C}$ and all common zeros of $M$ and the linear
remainder $R$ when dividing $C$ by $M$. Finding all zeros of $R$ is a linear
problem. Generically, it has a unique solution but zero or infinitely many
solutions are possible. A zero $h$ of $R$ is valid if it satisfies the quadratic
constraint $M(h) = 0$. This also ensures that $t-h$ is really a spinor
polynomial because it has the real norm polynomial~$M$.

\begin{remark}
  If the remainder polynomial $R=tr_1+r_0$ has an invertible leading coefficient
  (generic case), its zero is uniquely given by $h\coloneqq -r_0r_1^{-1}$. Thus
  we can write $R=(t-h)r_1$. As $M$ is a factor of $C\reverse{C}$, it follows
  from
  \begin{align*}
    C\reverse{C}=M^2Q\reverse{Q} + M(Q\reverse{R}+R\reverse{Q}) + R\reverse{R}
  \end{align*}
  that $M$ is also a factor of
  $R\reverse{R}=r_1\reverse{r_1}(t-h)(t-\reverse{h})$. Since $r_1$ is
  invertible, $h$ is also a zero of $M$, thus the quadratic constraint is
  automatically fulfilled.
\end{remark}

By recursively constructing linear factors in this way, we can compute all
decompositions of a spinor polynomial $C$ into linear factors.

\begin{itemize}
\item In each step of the outlined factorization procedure, the norm polynomial
  of the constructed left factor depends on the chosen quadratic real factor
  $M$. As multiplication is non-commutative, we will in general obtain different
  factorizations into linear factors, depending on a chosen order of quadratic
  factors of $C\reverse{C}$.
\item Existence of a factorization into linear factors is by no means
  guaranteed. It might happen, that the remainder polynomial $R$ is constant or
  that $M$ and $R$ have no common zeros.
\item Likewise, it is possible that $M$ and $R$ have infinitely many common
  zeros. This leads to spinor polynomials that admit infinitely many
  factorizations.
\end{itemize}

There is a fairly complete and recent a priori characterization of dual
quaternion polynomials, which admit a factorization
\cite{li22:_characterization}. A similar criterion for the factorizability of
split quaternion polynomials is not available and currently out of reach. For
this reason, we will largely focus on generic spinor polynomials in this article.

The first step in computing factorizations of $C$ according to the described
algebraic procedure consists of computing all monic, quadratic, real factors of
$C\reverse{C}$. This, essentially, means computing its zeros over $\C$ and
combining pairs of real or complex conjugate roots. In geometric terms computing
the complex roots of $C\reverse{C}$ amounts to computing the parameter values of
intersection points of $[C]$ with the null quadric~$\NQ$. In the remainder of
this section, we will show how to construct linear factors directly from these
points and their parameter values.

\subsection{Null Displacements}

We continue by exploring algebraic and kinematic properties of points in the
intersection of Study variety $\SV$ and null quadric $\NQ$. In doing so, we
freely use the scalar extension of $\CGAp$ by complex numbers. This does not
change essential properties of $\CGAp$ over $\R$ but affects sub-algebras such
as the quaternions or the dual quaternions. We are still interested in
real factorizations but complex algebra elements need to be considered in order
to get all real factorizations.

We call any point $[n] \in \SV \cap \NQ$ a \emph{null displacement.} The
following is an important lemma with a rather technical proof that we put off
into the appendix.

\begin{lemma}
  \label{lem:3}
  For a null displacement $[n]\in\SV\cap\NQ$ there exists a vector $x \neq 0$
  such that $xn=0$.
\end{lemma}

If no confusion with more general algebra elements is to be expected, we refer
to a vector $x$ as in Lemma~\ref{lem:3} as \emph{left annihilator.} By
Remark~\ref{rem:annihilator-point} below, it is even a point whence we also call
it \emph{left annihilating point.} Generically, it is unique up to scalar
multiples. For the following proposition, denote by $\SS$ the set of spheres
(including points and planes) in conformal three space.

\begin{proposition}
  \label{prop:2}
  For a null displacement $[n] \in \SV \cap \NQ$, the map $\SS \dashrightarrow
  \SS$, $[x] \mapsto [nx\reverse{n}]$ is constant wherever it is well-defined.
\end{proposition}

\begin{proof}
  Let $x$ be an arbitrary vector and set $y \coloneqq nx\reverse{n}$. By
  Lemma~\ref{lem:3}, there exists a vector $a$ such that $an=0$. Thus
  $ay=anx\reverse{n}=0$, which shows, that $y$ and $a$ are either linearly
  dependent, i.e. $[y] = [a]$ or $y=0$.
\end{proof}

\begin{remark}
  \label{rem:annihilator-point}
  The left annihilator $a$ of $n$ in Lemma~\ref{lem:3} satisfies $\reverse{a}a =
  0$ as we have $\reverse{a}an = 0$ and $n \neq 0$. Thus, it is a point,
  possibly with complex coordinates.
\end{remark}

\begin{remark}
  \label{rem:zero-displacement}
  For $[n] \in \SV \cap \NQ$ there exists a linear form $l(x)$ such that
  $nx\reverse{n} = l(x)a$. Take, for example $n = \eps_1$ whence $nx\reverse{n}
  = l(x)a$ with $l(x) = x_o$ and $a = e_\infty$. In non-generic cases, it is
  possible that $l$ vanishes. Thus, it is possible that the map mentioned in
  Proposition~\ref{prop:2} is nowhere defined. One example of this is $n =
  \eps_2(1 + \qi\ci)$.\footnote{The complex unit $\ci \in \C$ is not to be
    confused with the quaternion unit $\qi$.}
\end{remark}

The proof of Proposition~\ref{prop:2} shows:

\begin{corollary}
  For a null displacement $[n] \in \SV \cap \NQ$, the map $[x] \mapsto
  [nx\reverse{n}]$ is undefined on $\SS$ if and only if $n$ has two independent
  left annihilating points.
\end{corollary}

We see that there exist two types of null displacements: Generically, the left
annihilating points are unique up to scalar multiples and the map of
Proposition~\ref{prop:2} is well-defined for generic spheres. Special null
displacements of linearly independent left annihilators and the map $[x] \to
[nx\reverse{n}]$ is nowhere defined.

\subsection{A Geometric Factorization Algorithm}

Using the concept of null displacements and their annihilators, it is possible
to devise a predominantly geometric construction of left factors for a spinor
polynomial $C$ (and of course a symmetric construction of right factors). The
polynomial $C$ parametrizes a curve on the Study variety. This curve $[C]$
intersects the null quadric in a finite number of points. Let us take two
distinct intersection points $n_1$ and $n_2$ of this curve\footnote{The case of
  only a single intersection point of (necessarily) high multiplicity is dealt
  with in Section~\ref{sec:multiple-root}.} with the null quadric corresponding
to two (possibly complex conjugate) parameter values $z_1$ and $z_2$. Because of
Lemma~\ref{lem:1}, the remainder polynomial obtained from division of $C$ by the
real polynomial $(t-z_1)(t-z_2)$ is precisely the linear polynomial
interpolating $n_1$ at the parameter value $z_1$ and $n_2$ at $z_2$,
respectively. This suggests, that all the information about a left factor of $C$
(provided it exists) are encoded in $n_1$, $n_2$, $z_1$ and $z_2$. Indeed, the
following theorem shows how to construct a left factor from these entities
alone. It uses the concept of ``non-orthogonal points'' $a_1$, $a_2$. This means
that their Clifford dot product $a_1 \cdot a_2$ does not vanish.

\begin{theorem}
  \label{th:1}
  Let $C$ be a monic spinor polynomial and $M \coloneqq (t-z_1)(t-z_2)\in\R[t]$
  a monic, quadratic factor of $C\reverse{C}$ with $z_1$, $z_2 \in \C$. If
  $C(z_1)$ and $C(z_2)$ have non-orthogonal points $a_1$, $a_2$ as respective
  left annihilators, then $C$ has a left factor $t-h$ where
  \begin{equation}
    \label{eq:4}
    h\coloneqq z_1-\frac{z_1-z_2}{ 2a_1 \cdot a_2}a_1a_2
             = \frac{z_1+z_2}{2}-\frac{z_1-z_2}{ 2a_1\cdot a_2}a_1\wedge a_2
  \end{equation}
  and $(t-h)(t-\reverse{h})=M$.
\end{theorem}

\begin{proof}
  Let $a_1$ and $a_2$ be non-orthogonal points, such that $a_1C(z_1) = a_2C(z_2)
  = 0$. Such points exist by Lemma~\ref{lem:3}. Define $h \in \CGAp$ by
  \eqref{eq:4}. For $z_1$ and $z_2\in\R$, this $h$ obviously is a real element
  of $\CGAp$. For complex roots of $M$, we have $z_2=\overline{z_1}$. This
  implies $C(z_2)=\overline{C(z_1)}$ and we can choose $a_2=\overline{a}_1$.
  Writing $a_1=a_R + \ci a_I$ where $a_R$ is the real part and $a_I$ the
  imaginary part of $a_1$, we then have
  \begin{align*}
    a\overline{a} &= a_R^2 - a_I^2 - 2 \ci (a_R\wedge a_I),\\
    \overline{a}a &= a_R^2 - a_I^2 + 2 \ci (a_R\wedge a_I).
  \end{align*}
  Therefore, $2 a\wedge\overline{a}=a\overline{a}-\overline{a}a=-4\ci(a_R\wedge
  a_I)$. Since $z_1-\overline{z_1} \in \ci\R$, we see that $h$ is real in this
  case as well.

  We will show that $h$ is a left zero of $C$. Lemma~\ref{lem:2} then implies
  that $t-h$ is a left factor of $C$. We use polynomial division to write $C$ as
  $C=Q M + R$ with polynomials $Q$, $R \in \CGAp[t]$ and $\deg R \le 1$. By
  Lemma~\ref{lem:1}, it is sufficient to show that $h$ is a left zero of both,
  $R$ and $M=(t-z_1)(t-z_2)$. As $C(z_1)=R(z_1)$ and $C(z_2)=R(z_2)$, $R$ is the
  unique linear polynomial interpolating $n_1\coloneqq C(z_1)$ and $n_2\coloneqq
  C(z_2)$ at parameter values $z_1$ and $z_2$ respectively, i.e.
  \begin{equation*}
    R=\frac{(t-z_1)n_2-(t-z_2)n_1}{z_2-z_1}.
  \end{equation*}
  Moreover, we have
  \begin{align*}
    h-z_1 &= -\frac{z_1-z_2}{2 a_1 \cdot a_2} a_1 a_2,\\
    h-z_2 &= \frac{z_1-z_2}{ 2 a_1 \cdot a_2}(2 a_1 \cdot a_2 - a_1 a_2)=\frac{z_1-z_2}{2 a_1 \cdot a_2}a_2 a_1 .
  \end{align*}
  Since $a_1n_1=0$ and $a_2n_2=0$, it holds that $(h-z_1)n_2=0$ and $(h-z_2)n_1=0$.
  Thus $h$ is a left zero of $R$. Furthermore,
  \begin{equation*}
    M(h)=(h-z_1)(h-z_2)= -\Bigl( \frac{z_1-z_2}{2a_1 \cdot a_2} \Bigr)^2 a_1a_2a_2a_1 = 0.
    \qedhere
  \end{equation*}
\end{proof}

\begin{remark}
  By Equation~\eqref{eq:4}, the ``vector part'' $a_1 \wedge a_2$ of $h$ is, up
  to scalar multiples, determined by the intersection points $[n_1]$ and $[n_2]$
  alone and does not require knowledge of their respective parameter values
  $z_1$, $z_2$. This shows that the elementary motion $t - h$ is determined, up
  to an affine re-parametrization $t \mapsto \alpha t + \beta$ with $\alpha$,
  $\beta \in R$, by $[n_1]$ and $[n_2]$ alone. This observation is crucial for
  extensions of the factorization algorithm to algebraic motions that come
  without a given parametrization.
\end{remark}

\begin{remark}
  The left annihilating points $a_1$ and $a_2$ can be computed solving the
  systems of homogeneous linear equations arising from $a_1n_1 = a_2n_2 = 0$.
  Proposition~\ref{prop:2} suggests a more straightforward method: Pick a random
  vector $x$ and set $a_1 = n_1x\reverse{n}_1$, $a_2 = n_2x\reverse{n}_2$. It
  works, however, only in generic cases and for generic choices of~$x$.
\end{remark}

\begin{example}
  \label{ex:1}
  Let us consider the spinor polynomial $C = t^2+1+\eps_1(bt\qi+a\qj)$, where
  $0<b\le a$. It is actually a motion polynomial in the sense of
  \cite{hegedus13:_factorization2} and it is known that it admits factorizations
  (over the dual quaternions) if and only if $a = b$
  \cite{li15:_survey}*{Proposition~6}. Here, we investigate factorizability over
  \CGAp using Theorem~\ref{th:1}. There are only two intersection points of the
  curve $[C]$ with $\NQ$,
  \begin{equation*}
    n_1 = \eps_1(a \qj + b \ci \qi)
    \quad\text{and}\quad
    n_2 = \eps_1(a \qj - b \ci \qi).
  \end{equation*}
  If $a \neq b$, the unique left annihilating point for both, $n_1$ and $n_2$,
  is $a_1 = a_2 = e_\infty$. Hence, the necessary condition of
  Theorem~\ref{th:1} is not fulfilled. If, however, $a = b$, we have infinitely
  many respective left annihilating points,
  \begin{equation*}
    a_1 =  \mu_1\ci e_1 + \mu_1 e_2 + \lambda_1 e_\infty
    \quad\text{and}\quad
    a_2 = -\mu_2\ci e_1 + \mu_2 e_2 + \lambda_2 e_\infty
  \end{equation*}
  with $\mu_1$, $\mu_2$, $\lambda_1$, $\lambda_2 \in \C$. Via
  Theorem~\ref{th:1}, they give rise to infinitely many left factors $t-h_1$ and
  factorizations $C = (t - h_1)(t - h_2)$ where
  \begin{equation*}
    \begin{aligned}
      h_1 &= \frac{1}{2\mu_1\mu_2} (- 2 \mu_1 \mu_2 \qk + \eps_1((\lambda_1 \mu_2 + \lambda_2 \mu_1) \qi + (\lambda_1 \mu_2 - \lambda_2 \mu_1) \ci \qj)),\\
      h_2 &= \frac{-1}{2\mu_1\mu_2} (-2 \mu_1 \mu_2 \qk + \eps_1((2 a \mu_1 \mu_2 + \lambda_1 \mu_2 + \lambda_2 \mu_1) \qi +(\lambda_1 \mu_2 - \lambda_2 \mu_1) \ci \qj)).
    \end{aligned}
  \end{equation*}
  The quaternions $h_1$ and $h_2$ both have real coefficients if and only if
  $\mu_1$, $\mu_2$ and $\lambda_1$, $\lambda_2$, respectively, are complex
  conjugates. The factorizations we found are precisely those of
  \cite{li15:_survey}*{Proposition~15}. In other words, for this example
  extending the algebra from dual quaternions to \CGAp does not yield more
  factorizations.
\end{example}

The construction of the left factor in the proof of Theorem~\ref{th:1} uses two
non-orthogonal points $a_1$ and $a_2$ obtained from a quadratic factor of the
norm polynomial with \emph{distinct} roots. This raises the question, if it is
also possible to use a quadratic factor $(t-z)^2$ of the norm polynomial which
has one root $z$ of multiplicity two. This factor, however, corresponds to only
one intersection point $[n]=[C(z)]$ of the curve parametrized by a spinor
polynomial $C$ with the null quadric $\NQ$. Generically, $[n]$ will only have a
unique annihilating point $[a]$ and we cannot use the construction above. If the
left annihilating point of $[n]$ is not unique, there exist two distinct vectors
$a_1$ and $a_2$ such that $a_1n=a_2n=0$. This implies, however,
$(a_1a_2+a_2a_1)n=0$ and since $2a_1\cdot a_2= (a_1a_2+a_2a_1)\in\R$ lies in the
center of \CGA, we have $a_1\cdot a_2=0$. Thus, all possible choices of two
different left annihilating points of the same zero displacement are orthogonal
and cannot be used in the construction above. Therefore it is necessary to
investigate the case of quadratic factors of the norm polynomial separately.

\subsection{Norm Polynomials with Quadratic Factors}
\label{sec:multiple-root}

Theorem~\ref{th:1} only allows the construction of a left factor from two
\emph{distinct} intersection points of the curve $C$ with the null quadric. It
is, however, also possible, at least generically, to obtain a left factor from a
single intersection point $n$, provided it corresponds to a zero $z$ of
$C\reverse{C}$ of multiplicity two or higher. For this special case we are able
to provide a simple sufficient criterion for a left factor to exist that is also
\emph{necessary.}

\begin{theorem}
  \label{th:2}
  Let $C$ be a spinor polynomial such that $(t-z)^2\in\R[t]$ is a factor of
  $C\reverse{C}$. Then there exists a left factor $t-h$ of $C$ with
  $(t-h)(t-\reverse{h}) = (t-z)^2$ if and only if $\reverse{C'(z)}C(z) \neq 0$.
\end{theorem}

\begin{remark}
  Note that Theorem~\ref{th:2} only talks about real zeros $z$ of
  $C\reverse{C}$. Including complex zeros would be possible but result in a left
  factor $t-h$ where $h$ has complex coefficients -- something we generally wish
  to avoid. Also note that a complex zero can always be paired with its complex
  conjugate to provide suitable input data for the factorization according to
  Theorem~\ref{th:1}.
\end{remark}

We will prove Theorem~\ref{th:2} by reducing the statement to the special case
where $e_\infty C(z)=0$. This can always be done via a conformal transformation:
From Lemma~\ref{lem:3} we know there exists a vector $a$, such that $aC(z)=0$.
If $a \neq e_\infty$, the vector $a+e_\infty$ does not square to zero and hence
represents an invertible transformation. We can thus study the polynomial
$(a+e_\infty)C(a+e_\infty)$ instead of $C$. Indeed, because of $aC(z) = 0$ we
have
\begin{equation*}
  e_\infty(a + e_\infty)C = e_\infty a C(z) = 0,
\end{equation*}
so that the polynomial $(a + e_\infty)C(z)(a+e_\infty)$ fulfills the desired
property. This allows us to make simplifying assumptions on $C(z)$ by the
following lemma.

\begin{lemma}
  \label{lem:4}
  Let $n \in \CGAp$ be such that $e_\infty n=0$. Then $n=\eps_1(h_1+\eps_2 h_2)$
  for two quaternions $h_1$, $h_2\in\H$.
\end{lemma}

We omit the proof of Lemma~\ref{lem:4} as it consists of a straightforward
computation.

\begin{proof}[Proof of Theorem~\ref{th:2}]
  We assume that $z=0$ and $e_\infty C(0) = 0$. Neither of these assumptions is
  a loss of generality. The former can be achieved by a simple
  re-parametrization, the latter by the considerations preceding
  Lemma~\ref{lem:4}.

  Let us define $c_0 \coloneqq C(0)$ and $c_1 \coloneqq C'(0)$ so that the
  remainder polynomial when dividing $C$ by $t^2$ equals $R=c_1t+c_0$. In order
  to find a left factor $t-h$ of $C$, we need to find $h$ such that $R(h)=0$ and
  $(t-h)(\reverse{t-h})=t^2$. From \cite{dorst16}, we know that there are only
  two types of elementary motions which have a norm polynomial with a root of
  multiplicity two, translations and transversions. Both are given by $t-h$
  where $h= ab$ for a point $a$ and a \emph{plane $b$.} Assuming $t-h$ is a left
  factor of $C$ we further get $ac_0=0$. But by our initial assumption, we also
  have $e_\infty c_0=0$, which then implies $(ae_\infty+e_\infty a)c_0=0$ and
  therefore $a\cdot e_\infty =0$. But the only real point fulfilling this
  identity is $e_\infty$ itself and therefore we obtain $a=e_\infty$. Thus, we
  need to find $h=e_\infty b$ for a plane $b = b_1e_1+b_2e_2+b_3e_3+b_\infty
  e_\infty$ such that $R(h)=0$.

  By assumption we have $e_\infty C(0) = e_\infty c_0 = 0$. Lemma~\ref{lem:4}
  shows that there exist quaternions $q_1$, $q_2\in\H$ such that
  $c_0=\eps_1(q_1+\eps_2 q_2)$. Since $C$ is a spinor polynomial, $c_0= C(0)$
  fulfills the Study conditions, which, in this case, simplify to the single
  condition $\SC(q_1,q_2) = 0$. With this, the condition $R(h)=0$ becomes
  \begin{equation}
    \label{eq:5}
    \begin{aligned}
      0 &= e_\infty bc_1+\eps_1(q_1+\eps q_2)\\
        &= e_\infty(bc_1-e_{123}(q_1+\eps_2 q_2))\\
        &= e_\infty e_{123}(-e_{123}bc_1-(q_1+\eps_2 q_2))\\
        &=-\eps_1((b_1\qi+b_2\qj+b_3\qk)c_1-(q_1+\eps_2 q_2)).
    \end{aligned}
  \end{equation}
  Let us denote the vectorial quaternion $b_1\qi+b_2\qj+b_3\qk$ by $B$. By
  Lemma~\ref{lem:4}, Equation~\eqref{eq:5} is fulfilled, if and only if there
  exist quaternions $h_1$, $h_2\in\H$ such that
  \begin{equation*}
    Bc_1-(q_1+\eps_2 q_2)=\eps_1(h_1+\eps_2 h_2).
  \end{equation*}
  This is the case, if the sum of the coefficients of $1$ and $\eps_3$ as well
  as the coefficient of $\eps_2$ vanish in the four-quaternion representation of
  $Bc_1-(q_1+\eps_2 q_2)$. Using the four-quaternion representation
  $c_1=r_0+\eps_1 r_1 +\eps_2 r_2+ \eps_3 r_3$, this is equivalent to the two
  quaternionic equations
  \begin{equation}
    \label{eq:6}
    Br_0-q_1=0,\qquad Br_2-q_2=0
  \end{equation}
  for the unknown $B$. Either of these equations has a unique solution for $B$
  provided both $r_0$ and $r_2$ are different from $0$, i.e. $B_1=q_1r_0^{-1}$
  and $B_2=q_2r_2^{-1}$. We have to show that both solutions are non-trivial,
  coincide and are vectorial quaternions. To see this, we use the multiple root
  condition $(\reverse{C}C)^\prime(0)=\reverse{c_0}c_1+\reverse{c_1}c_0=0$.
  Again invoking the four-quaternion representations for $c_0$ and $c_1$,
  respectively, this yields four quaternionic conditions:
  \begin{align}
    \reverse{q_2}r_0+\reverse{r_2}q_1 &=0, \label{eq:7}\\
    \reverse{q_1}r_2+\reverse{r_0}q_2&=0, \label{eq:8}\\
    \SC(q_1,r_0)&=0, \label{eq:9}\\
    \SC(q_2,r_2)&=0. \label{eq:10}
  \end{align}
  Equations \eqref{eq:9} and \eqref{eq:10} ensure, that both $B_1$ and $B_2$ are
  vectorial quaternions. To see, that they are the same, let us take $B_1$ and
  plug it into the second equation of \eqref{eq:6}. Multiplying away the
  denominator $r_0\reverse{r}_0$ of $B_1$ and using Eq.~\eqref{eq:8} and
  Eq.~\eqref{eq:9}, this yields
  \begin{equation*}
    q_1\reverse{r_0}r_2-r_0\reverse{r_0}q_2 = -r_0\reverse{q_1}r_2 +r_0\reverse{q_1}r_2=0.
  \end{equation*}
  Thus, $B_1$ and $B_2$ are equal. Further, this solution is non-trivial, as we
  assumed $r_0$ and $r_2$ to be non-zero and further, $q_1$ and $q_2$ cannot
  vanish simultaneously as otherwise $c_0$ is zero and consequently, $C$ not
  reduced.

  Now let us consider the case, where either $r_0$ or $r_2$ is zero, but they do
  not vanish simultaneously. Then Eq.~\eqref{eq:7} ensures, that one of the
  equations in \eqref{eq:6} is already fulfilled and we can take the solution of
  the other one.

  Finally, from $r_0=r_2=0$ it follows $c_1=\eps_1(r_1+\eps_2 r_3)$ which is a
  contradiction to our assumption $\reverse{c_1}c_0\neq 0$.
\end{proof}

\begin{remark}
  For the case of \emph{split quaternions}, a sub-algebra of \CGAp, the article
  \cite{scharler21} gives a necessary and sufficient criterion for existence of
  a linear left factor that corresponds, via Theorem~\ref{th:1}, to two
  intersection points $[C(z_1)]$, $[C(z_2)]$ of $[C]$ and $\NQ$. It
  reads
  \begin{equation}
    \label{eq:11}
    \reverse{C(z_1)}C(z_2) \neq 0
  \end{equation}
  and we may view the condition of Theorem~\ref{th:2} as a limiting version of
  \eqref{eq:11}. This suggests that \eqref{eq:11} may be a sufficient and
  \emph{necessary} condition also in the case of spinor polynomials.
  Unfortunately, this is not true which can be seen in the following
  example.
\end{remark}

\begin{example}
  \label{ex:2}
  In Example~\ref{ex:1}, we considered the spinor polynomial $C=t^2+1+\eps_1(b t
  \qi + a \qj)$ which admit (infinitely many) factorizations if and only if
  $a=b$. The two intersection points of $[C]$ with $\NQ$ are $[n_1]$ and $[n_2]$
  where
  \begin{equation*}
    n_1=\eps_1(a\qj+bi\qi)\quad\text{and}\quad n_2=\eps_1(a\qj-bi\qi).
  \end{equation*}
  As $\eps_1^2=0$, it holds $\reverse{n_1}n_2=0$, regardless of the existence of
  a factorization.
\end{example}

\section{A Multiplication Technique for Factorizing Spinor Polynomials}
\label{sec:factorization-technique}

As remarked in Section~\ref{sec:geometric-factorization-algorithm}, not every
spinor polynomials $C$ admit a factorization. For the sub-algebras of dual and
split quaternions of $\CGAp$, there exist a multiplication techniques which
allows to obtain a factorization of $RC$ where $R \in \R[t]$ is a suitable real
polynomial. For applications in kinematics
\cites{gallet16,li18:_universality_theorem,lercher22:_multiplication_technqiue},
this is important as $C$ and $CR$ represent the same rational motion. While the
real co-factor $R$ in case of dual quaternions is quite tricky to compute
\cite{li19}, a generic real polynomial of suitable degree will do in case of
split quaternions \cite{scharler21}. The same is true for spinor polynomials and
a proof will be given in this section.

\begin{theorem}
  \label{th:3}
  Let $P$ be a spinor polynomial which does not admit a left or a right factor.
  Then there exists a spinor polynomial $H \coloneqq t-h$ such that $C\coloneqq
  PH$ admits both a left and a right factor.
 \end{theorem}
 \begin{proof}
   Denote the zeros of $H\reverse{H}$ by $z_1$ and $z_2$. Further, let us first
   assume that the norm polynomial $P\reverse{P}$ has two distinct roots $t_1$
   and $t_2$. Let $b_1$ and $b_2$ be right annihilators of $P(t_1)$ and
   $P(t_2)$, respectively. To prove, that the polynomial $C$ admits a left and a
   right factor, we need to show, that $C(t_1)$ and $C(t_2)$ have non-orthogonal
   right annihilators, and $C(z_1)$ and $C(z_2)$ have non-orthogonal left
   annihilators for an appropriate choice of~$H$. In addition, $H$ has to
   fulfill the Study conditions.

   We choose two non-orthogonal vectors $e$ and $f$ and define $H$ as the
   interpolation polynomial of $ef$ and $-fe$, i.e. $H=t+e\wedge f$. We have six
   essential degrees of freedom to choose $e$ and $f$ but need to avoid
   orthogonality, i.e. one quadratic condition. For further arguments, we
   additionally need to ensure, that the roots of $H\reverse{H}$ are different
   from any roots of $P\reverse{P}$, which gives a finite number of additional
   algebraic conditions to avoid.

   As $P(t_i)b_i=0$ and $H(t_i)\reverse{H(t_i)}\in\C\setminus\{0\}$, it holds
   \begin{equation*}
     C(t_i)\reverse{H(t_i)}b_iH(t_i)=P(t_i)H(t_i)\reverse{H(t_i)}b_iH(t_i)=0,
   \end{equation*}
   for $i \in \{1, 2\}$. Thus, we have found suitable right annihilators $r_1 =
   \reverse{H(t_1)}b_1H(t_1)$ and $r_2 = \reverse{H(t_2)}b_2H(t_2)$ of $C(t_1)$
   and $C(t_2)$ respectively. Similar arguments show, that $l_1\coloneqq
   P(z_1)e\reverse{P(z_1)}$ is a left annihilator of $C(z_1)$ and $l_2\coloneqq
   P(z_2)f\reverse{P(z_1)}$ is a left annihilator of $C(z_2)$.

   As we want to avoid orthogonality of these left and right annihilators, we
   need to fulfill one condition of degree four and one of degree two on the
   coefficients of $H$. In summary, we need to avoid a finite number of
   algebraic sets of dimension at most five which is certainly possible.

   We still need to consider the case, where $P\reverse{P}$ has only one root
   $z\in\R$ of multiplicity at least two. Let $b_1$ be a left annihilating point
   of $P(z)$ and $b_2$ a right annihilating point of $P(z)$. Further, let us
   define $a_1=P(z_1)e\reverse{P(z_1)}$ for some vector $e$ and
   $a_2=\reverse{H(z)}b_2H(z)$. It is straightforward to see, that $a_1$ is a
   left annihilator of $C(z_1)$ and $b_1$ a left annihilator of $C(z)$. Under
   the condition that they are not orthogonal, we can use Theorem~\ref{th:1} to
   construct a left factor of $C$. Similarly, $a_2$ is a right annihilator of
   $C(z)$, and $e$ is a right annihilator of $C(z_2)$. Thus we can use
   Theorem~\ref{th:1} to find a right factor of $C$. Again, this is possible if
   we avoid a finite number of low dimensional algebraic sets.
 \end{proof}

\begin{corollary}
  For any spinor polynomial $P \in \CGAp[t]$ there exists a spinor polynomial $H
  \in \CGAp[t]$ such that $PH$ admits a factorization with linear factors.
  Moreover, there exists a real polynomial $R \in \R[t]$ such that $PR$ admits a
  factorization with linear factors.
\end{corollary}

\begin{proof}
  The first statement follows by induction on the degree of $P$ from
  Theorem~\ref{th:3}. By Theorem~\ref{th:3}, the co-factor $H$ admits a
  factorization with linear factors. Thus, the second statement follows from the
  first with $R = H\reverse{H}$.
\end{proof}

\begin{example}
  \label{ex:3}
  The polynomial $t^2+\eps_3$ has the norm polynomial $(t^2+1)(t^2-1)$. After
  division by either of its two real quadratic factors, the respective remainder
  polynomials, $\eps_3+1$ and $\eps_3-1$, are constant and have no zero. Thus,
  $t^2+\eps_3$ does not admit a factorization into two linear factors by
  Lemma~\ref{lem:2}.

  However, with vectors $e = e_1 + e_o$ and $f = e_2 + e_\infty$, we have $e
  \cdot f = -1 \neq 0$, $H = t - e \wedge f = t + \qk - \qi\eps_1 + \qj\eps_2 +
  \qk\eps_3$ and $(t^2 + \eps_3)H = (t-h_1)(t-h_2)(t-h_3)$ where
  \begin{equation*}
    \begin{aligned}
      h_1 &= -\qk - \qi\eps_1 + \qj\eps_2 - \eps_3, \\
      h_2 &= \phantom{-}\qk + (\qi + \tfrac{1}{2}\qj)\eps_1 - \qj\eps_2 + \eps_3, \\
      h_3 &= -\qk + (\qi - \tfrac{1}{2}\qj)\eps_1 - \qj\eps_2 - \eps_3.
    \end{aligned}
  \end{equation*}
  There is nothing particular about the vectors $e$ and $f$. Any generic choice
  will do.
\end{example}

\section{Factorization of an Algebraic Four-Bar Motion}
\label{sec:four-bar}

This section demonstrates, at hand of a single example, that ideas of our
geometric factorization algorithm generalize to algebraic motions. Its purpose
is twofold: It serves as a motivation for introducing an alternative
factorization algorithm for spinor polynomials and it provides an outlook to
future research. The reader is kindly asked to view the contents of this section
as a proof of concept and accept that some steps in the computation only come
with a rather vague justification. The actual factorization theory of algebraic
motions is yet to be worked out.

Let us consider the algebraic curve $C$ in the projective space $\P(\H) =
\P^3(\R)$ over the quaternions $x = x_0 + x_1\qi + x_2\qj + x_3\qk$ that is
given by the ideal
\begin{multline}
  \label{eq:12}
  \langle
  x_0^2+x_1^2+x_2^2+x_3^2-4(x_0x_2+x_1x_3),\\
  13x_0^2-3x_1^2+13x_2^2-3x_3^2+16(x_0x_1+x_2x_3)-12(x_0+x_1)(x_2-x_3)
  \rangle.
\end{multline}
It is of genus one. The two generating polynomials were obtained as ``circle
constraint equations'' in the sense of \cite{brunnthaler06}{Section~3.1]. They
encode the condition that two unit vectors $m_1$, $m_2$ in the moving coordinate
frame are mapped, via rotations described by \eqref{eq:12}, to respective
circles with normalized unit axes $f_1$, $f_2$ in the fixed system. The motion
given by $C$ can be mechanically realized as coupler motion of a spherical
four-bar linkage. It is our aim, to compute the fixed axes $f_1$, $f_2$ and the
moving axes $f_1$, $f_2$ using our geometric factorization algorithm.

By abuse of notation, we denote by $\NQ$ the intersection of the null quadric
with the projective space $\P(\H) = \P^3(\R)$ of quaternions and we refer to
this intersection by ``null quadric'' as well. It is a regular and doubly ruled
quadric. In a first step, we intersect the curve $C$ and the null quadric $\NQ$.
Since the curve is of degree four, we expect eight intersection points. Indeed,
we find
\begin{equation}
  \label{eq:13}
  \begin{aligned}
    [n_1] &= [2 - 2\ci\qi + \sqrt{2}((1+\ci)\qj - (1-\ci)\qk)],\\
    [n_2] &= [2 - 2\ci\qi - \sqrt{2}((1+\ci)\qj - (1-\ci)\qk)], \\
    [n_3] &= [5 - (4 + 3\ci)\qi + (3 - 4\ci)\qj + 5\ci\qk], \\
    [n_4] &= [5 + (4 + 3\ci)\qi - (3 - 4\ci)\qj + 5\ci\qk],
  \end{aligned}
\end{equation}
and the respective complex conjugates $[\overline{n}_1]$, $[\overline{n}_2]$,
$[\overline{n}_3]$, $[\overline{n}_4]$. It will later be important that the
point pairs
\begin{equation*}
  (n_1, \overline{n}_2),\quad
  (n_2, \overline{n}_1),\quad
  (n_3, \overline{n}_4),\quad
  (n_4, \overline{n}_3)
\end{equation*}
lie on rulings of the first kind on $\NQ$, while the point pairs
\begin{equation}
  \label{eq:14}
  (n_1, n_2),\quad
  (\overline{n}_1, \overline{n}_2),\quad
  (n_3, n_4),\quad
  (\overline{n}_3, \overline{n}_4)
\end{equation}
lie on rulings of the second kind.

Left annihilating points to $n_1$, $n_2$, $n_3$ and $n_4$ are
\begin{equation*}
  a_1= 2e_1 + \sqrt{2}\ci(e_2 - e_3),\quad
  a_2= 2e_1 - \sqrt{2}\ci(e_2 - e_3),\quad
  a_3= e_1 - \ci e_2,\quad
  a_4= e_1 + \ci e_2,
\end{equation*}
respectively. They come in conjugate complex pairs, even if they did not arise
from conjugate complex null displacements. This is no coincidence as we only
expect two real left factors. In fact, points on the same ruling of the second
kind yield the same left annihilators. The two revolute axes in the fixed frame
are given by the wedge products
\begin{equation*}
  a_1 \wedge \overline{a}_1 = -a_2 \wedge \overline{a}_2 = 4\sqrt{2}\ci(\qj+\qk),
  \quad
  a_3 \wedge \overline{a}_3 = -a_4 \wedge \overline{a}_4 = -2\ci\qk.
\end{equation*}
After normalization, we find the axis directions $f_1 =
\frac{1}{\sqrt{2}}(\qj+\qk)$ and $f_2 = \qk$, respectively.

The right factors will give rise to the moving axes. Their computation is based
on the null points \eqref{eq:13} as well but, in view of \eqref{eq:14}, we
compute right annihilating points of $n_1$, $\overline{n}_2$, $n_3$, and
$\overline{n}_4$:
\begin{equation*}
    b_1 = e_2 + \ci e_3,\quad
    \overline{b}_2 = e_2 - \ci e_3,\quad
    b_3 = 4e_1 - 3e_2 + 5\ci e_3,\quad
    \overline{b}_4 = 4e_1 - 3e_2 - 5\ci e_3.
\end{equation*}
Here, points on the same rulings of first kind give identical annihilators. The
moving axes directions are found as
\begin{equation*}
  b_1 \wedge \overline{b}_1 = b_2 \wedge \overline{b}_2 = 2\ci\qi,
  \quad
  b_3 \wedge \overline{b}_3 = b_4 \wedge \overline{b}_4 = -10\ci(3\qi + 4\qj).
\end{equation*}
Unit direction vectors of the two moving axes in the moving frame are $m_1 =
-\frac{1}{5}(3\qi + 4\qj)$ and $m_2 = -\qi$, respectively (the choice of sign is
irrelevant but helps to make the visualization of the mechanism more compact).
Mapping these two points via transformations in $C$ produces the spherical
four-bar motion. The underlying mechanical structure is illustrated in
Figure~\ref{fig:1}.

\begin{figure}
  \centering
\includegraphics{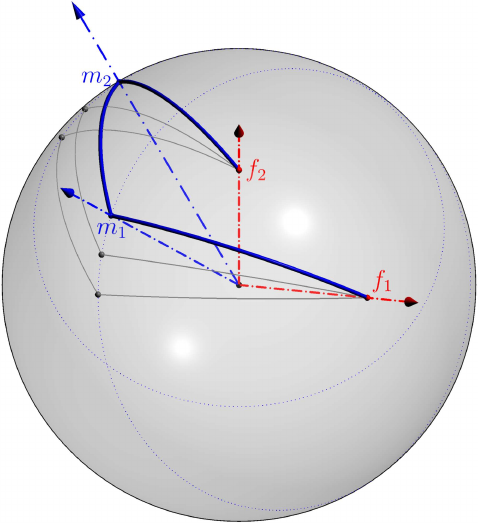}
  \caption{A spherical four-bar linkage.}
  \label{fig:1}
\end{figure}

The seaming ease of these computations is deceiving. In case of four-bar
linkages with axes only in the fixed and in the moving frame, everything is
straightforward indeed. If further moving axes are involved, additional
considerations are required:
\begin{itemize}
\item A proper mathematical definition of left and right factors is needed.
\item It seems already clear that, in contrast to rational motion, left and
  right factors of algebraic motions generically don't exist. Formulation of
  necessary and maybe even sufficient criteria for their existence would be
  useful.
\item For algebraic motions with three or more factors, it will be necessary to
  ``split off'' a left/or right factor by a non-obvious algebraic version of
  polynomial division.
\item One needs to ensure that all axes are computed with respect to a
  particular configuration of the linkage and also with respect to the same
  coordinate frame.
\end{itemize}
The latter point is exemplified by above example: The axes $f_1$ and $f_2$ are
fixed in the fixed frame, the axes $m_1$, $m_2$ are fixed in the moving frame.
In spite of the labeling, Figure~\ref{fig:1} does not display $m_1$ and $m_2$
directly but their images in the fixed frame for particular configurations of
the linkage.

Working out a general factorization theory for algebraic motions will be left to
future publications.

A further line of future research concerns generalizations of our results to
conformal algebras of dimension larger than three or even to more general
Clifford algebras. In fact, with exception of Lemma~\ref{lem:3}, none of the
arguments depends crucially on the dimension of the underlying conformal space.
The formulation of Lemma~\ref{lem:3} makes sense in other Clifford algebras as
well but its rather technical proof in the appendix relies on arguments that
only hold for conformal geometric algebra in dimension three.

\section*{Acknowledgment}

Johannes Siegele was supported by Austrian Science Fund (FWF) P~33397-N (Rotor
Polynomials: Algebra and Geometry of Conformal Motions).

\section{Appendix}

\begin{table}
  \centering
  \caption{Multiplication table for $e_0$, $e_{123}$ and $e_\infty$ with
    $\eps_1$, $\eps_2$ and $\eps_3$.}
  \label{tab:eps-multiplication2}
  \begin{tabular}{c|ccc}
    & $\eps_1$      & $\eps_2$     & $\eps_3$  \\
    \hline
    $e_0$ & $e_{0123\infty}$  & $0$ & $-e_0$  \\
    $e_{123}$ & $-e_\infty$ & $-e_0$          & $e_{0123\infty}-e_{123}$ \\
    $e_\infty$ & $0$     & $2e_{123}-e_{0123\infty} $     & $e_\infty$
  \end{tabular}
\end{table}

\begin{proof}[Proof of Lemma~\ref{lem:3}]
  Let us write $n$ in the four quaternion representation, i.e. $n =
  q_0+q_1\eps_1+q_2\eps_2+q_3\eps_3$, and $x = x_0e_0+Xe_{123}+x_\infty
  e_\infty$, where $X=x_1\qi+x_2\qj+x_3\qk\in\H$ is a vectorial quaternion.
  Using the multiplication Table~\ref{tab:eps-multiplication2} we get
  \begin{equation}
    \begin{aligned}
      \label{mult:1}
      xn &= (x_0e_0+Xe_{123}+x_\infty e_\infty)(q_0+q_1\eps_1+q_2\eps_2+q_3\eps_3)\\
         &= e_0 (x_0(q_0-q_3)-Xq_2)\\
         &\phantom{=}+e_\infty(x_\infty(q_0+q_3)-Xq_1)\\
         &\phantom{=}+e_{123}(X(q_0-q_3)+2x_\infty q_2)\\
         &\phantom{=}+e_{0123\infty}(x_0q_1-x_\infty q_2+Xq_3).
    \end{aligned}
  \end{equation}
  For this product to equal zero, we need that the coefficients of $e_0$,
  $e_\infty$, $e_{123}$ and $e_{0123\infty}$ in Equation~\eqref{mult:1} vanish.
  This gives rise to four equations over quaternions:
  \begin{align}
    x_0(q_0-q_3)-Xq_2=0\label{eq:sys1}\tag{\scshape{i}}\\
    x_\infty(q_0+q_3)-Xq_1=0\label{eq:sys2}\tag{\scshape{ii}}\\
    X(q_0-q_3)+2x_\infty q_2=0\label{eq:sys3}\tag{\scshape{iii}}\\
    x_0q_1-x_\infty q_2+Xq_3=0\label{eq:sys4p}\tag{\scshape{iv}$\prime$}
  \end{align}
  For convenience, let us replace \eqref{eq:sys4p} by
  2\eqref{eq:sys4p}+\eqref{eq:sys3} which results in the new fourth equation
  \begin{equation}
    \label{eq:sys4}
    \tag{\scshape{iv}}
    X(q_0+q_3)+2x_0q_1=0.
  \end{equation}
  The aim is to find a non-zero solution for this system of quaternion
  equations. We will do so successively by explicitly giving a solution and
  subsequently discussing the cases where this solution is zero.

  \paragraph{Case 1:}
  We give two explicit solutions. They span a subspace of the vector space of
  all solutions. For a slightly more compact notation, let us define $\norm{x}
  \coloneqq x\reverse{x}$.

  \paragraph{Solution 1.1:} The first solution is given by
  $X=q_2\reverse{(q_0-q_3)}=-(q_0-q_3)\reverse{q_2}$. Plugging this into
  Equations~\eqref{eq:sys1}--\eqref{eq:sys4} we obtain
  \begin{align*}
    (x_0+\norm{q_2})(q_0-q_3)                   & =0 \\
    x_\infty(q_0+q_3)+(q_0-q_3)\reverse{q_2}q_1 & =0 \\
    (\norm{q_0-q_3}+2x_\infty)q_2               & =0 \\
    q_2\reverse{(q_0-q_3)}(q_0+q_3)+2x_0q_1     & =0.
  \end{align*}
  The choice $x_0=-\norm{q_2}$ solves the first equation and $x_\infty =
  -\norm{q_0-q_3}/2$ solves the third equation. To see that this is a
  solution for the whole system, let us at first recall the following: The point
  $[n]$ lies on the null quadric whence $\norm{q_0}-\norm{q_3}-\SC(q_1,q_2)=0$.
  Using this and the Study condition
  $\Vect(q_0\reverse{q_3})=\Vect(q_1\reverse{q_2})$ we obtain
  \begin{equation}
    \label{eq:15}
    \begin{aligned}
      (q_0+q_3)\reverse{(q_0-q_3)} &= \norm{q_0}-\norm{q_3} -2\Vect(q_0\reverse{q_3})\\
                                   &= \SC(q_1,q_2)-2\Vect(q_1\reverse{q_2})\\
                                   &= 2q_2\reverse{q_1}.
    \end{aligned}
  \end{equation}
  Similarly, we will get
  \begin{equation}\label{eq:16}
    \reverse{(q_0+q_3)}(q_0-q_3)=2\reverse{q_1}q_2.
  \end{equation}
  Using Equations~\eqref{eq:15} and \eqref{eq:16} it is easy to see, that we
  indeed found a solution for our system.

  \paragraph{Solution 1.2:}
  Let us set $X=q_1\reverse{(q_0+q_3)}=-(q_0+q_3)\reverse{q_1}$. Again using
  Equations~\eqref{eq:15} and \eqref{eq:16}, the system of equations reads
  \begin{align*}
    (x_0+\tfrac{1}{2}\norm{q_0+q_3})(q_0-q_3)&=0\\
    (2\norm{q_1}+2x_\infty)q_2&=0\\
    (\norm{q_0+q_3}+2x_0)q_1&=0
  \end{align*}
  which obviously admits the solution $x_0=-\norm{q_0+q_3}/2$,
  $x_\infty=-\norm{q_1}$.\footnote{It can be shown that Solutions~1.1 and 1.2
    are linearly dependent. The purpose of stating them separately is that in
    our further discussion we can assume that both of them vanish which gives us
    more conditions to work with.}

  \paragraph{Case 2:}
  It might still be the case, that any linear combination of the solutions given
  above is zero, i.e. $(q_0-q_3)\reverse{q_2}=(q_0+q_3)\reverse{q_1}=0$ and
  $\norm{q_0+q_3}=\norm{q_0-q_3}=\norm{q_1}=\norm{q_2}=0$. From these
  assumptions, we immediately get $\SC(q_0,q_3)=0$ and $\norm{q_0}=-\norm{q_3}$
  which simplifies the null quadric condition to $2\norm{q_0}=\SC(q_1,q_2)$.
  Again we will give two solutions for this case. For this, we need the
  following property for zero divisors in the algebra of quaternions: For two
  quaternions $x$, $y\in\H$ with $\norm{x}=0$ it holds
  \begin{equation}
    \label{eq:17}
    x\reverse{y}x=(x\reverse{y}+y\reverse{x})x=\SC(x,y)x.
  \end{equation}

  \paragraph{Solution 2.1:}
  Let us choose $X=q_2\reverse{(q_0+q_3)}=-(q_0+q_3)\reverse{q_2}$.
  Equations~\eqref{eq:sys1} and \eqref{eq:sys4} suggests $x_0=0$ as a solution.
  Putting $x_\infty = -2 \norm{q_0}$ and using Equations~\eqref{eq:15} and
  \eqref{eq:16}, Equation~\eqref{eq:sys2} reads
  \begin{align*}
    x_\infty(q_0+q_3)+(q_0+q_3)\reverse{q_2}q_1&= \bigl(x_\infty+\tfrac{1}{2}\SC(q_0+q_3,q_0-q_3)\bigr)(q_0+q_3) \\
                                               &= \bigl(x_\infty+\SC(q_1,q_2)\bigr)(q_0+q_3)\\
                                               &=(x_\infty+2\norm{q_0})(q_0+q_3)\\
                                               &=0,
  \end{align*}
  and Equation\eqref{eq:sys3} reads
  \begin{align*}
    q_2\reverse{(q_0+q_3)}(q_0-q_3)+2x_\infty q_2 &= 2q_2\reverse{q_1}q_2+2x_\infty q_2\\
                                                  &= 2(\SC(q_1,q_2)+x_\infty)q_2\\
                                                  &= 2(2\norm{q_0}+x_\infty)q_2\\
                                                  &=0.
  \end{align*}
  Thus, we have really found a solution for this case.

  \paragraph{Solution 2.2:}
  Let us choose $X=q_1\reverse{(q_0-q_3)}=-(q_0-q_3)\reverse{q_1}$. From
  Eqs.~\eqref{eq:sys2} and \eqref{eq:sys3} we obtain $x_\infty =0$. Similar to
  the solution above, we can show that $x_0=-2\norm{q_0}$ gives a solution for
  our system of equations.

  \paragraph{Case 3:}
  In this case, the construction of the solution is a bit more intricate. For
  now, we will assume, that $q_1\neq 0$, $q_2\neq 0$, $q_0+q_3\neq 0$ and
  $q_0-q_3\neq 0$. Again we need to consider the case, where every linear
  combination of Solution 2.1 and Solution 2.2 is zero. This condition yields
  $\norm{q_0}=\norm{q_1}=\norm{q_2}=\norm{q_3}=0$ and further
  \begin{equation*}
    (q_0-q_3)\reverse{q_2}=(q_0+q_3)\reverse{q_2}=(q_0+q_3)\reverse{q_1}=(q_0-q_3)\reverse{q_1}=0.
  \end{equation*}
  Denoting, once more, the intersection of null quadric and $\P(\H) = \P^3(\R)$
  by $\NQ$, above equations show that $[q_0+q_3]$ and $[q_0-q_3]$ lie on a
  ruling of $\NQ$ through $[q_1]$ and on a ruling of $\NQ$ through $[q_2]$.
  Thus, these rulings must coincide and therefore all four projective points in
  $\P(\H)$ lie on a common ruling. Through each of these points passes a ruling
  of the other (second) family of rulings. They intersect the plane of vectorial
  quaternions. They can be obtained as $[U_1]$,
  $[U_2]$, $[U_3]$, and
  $[U_4]$ where
  \begin{equation*}
    U_1 \coloneqq q_1r\reverse{q_1},\quad U_2\coloneqq q_2r\reverse{q_2},\quad U_3\coloneqq (q_0+q_3)r\reverse{(q_0+q_3)},\quad U_4\coloneqq (q_0-q_3)r\reverse{(q_0-q_3)},
  \end{equation*}
  and $r$ is an arbitrary quaternion such that $U_1$, $U_2$, $U_3$, and
  $U_4$ are all different from zero.

  \paragraph{Case 3.1:}
  Let us assume, that the $U_i$, $i=1,\ldots,4$ are pairwise linearly
  independent. Then there exist coefficients $\lambda_1$, $\lambda_2$,
  $\lambda_3$, $\lambda_4\in\C$ such that
  \begin{equation}
    \label{eq:18}
    \lambda_1 U_1 + \lambda_3 U_3 =
    \lambda_2 U_2 + \lambda_4 U_4.
  \end{equation}

  To solve our system of equations, we will choose $X=\lambda_1 U_1 + \lambda_3
  U_3 = \lambda_2 U_2 + \lambda_4 U_4$. Plugging this into
  Equation~\eqref{eq:sys1} and using a Study condition as well as \eqref{eq:17}, we obtain
  \begin{align*}
    x_0(q_0-q_3)-\lambda_4U_4q_2&=x_0(q_0-q_3)-\lambda_4(q_0-q_3)r\reverse{(q_0-q_3)}q_2\\
                                &=x_0(q_0-q_3)+\lambda_4(q_0-q_3)r\reverse{q_2} (q_0-q_3)\\
                                &=(q_0-q_3)(x_0+\lambda_4 \SC(q_0-q_3,q_2 \reverse{r})).
  \end{align*}
  This suggests the solution $x_0=-\lambda_4\SC(q_0-q_3,q_2\reverse{r})$. In the
  same way, Equation~\eqref{eq:sys4} suggests the solution
  $x_0=\lambda_1\SC(q_1,(q_0+q_3)\reverse{r})/2$. Equations~\eqref{eq:sys2} and
  \eqref{eq:sys3} lead us to $x_\infty =-\lambda_3\SC(q_0+q_3,q_1\reverse{r})$
  and $x_\infty = \lambda_2\SC(q_2,(q_0-q_3)\reverse{r})/2$, respectively.

  To see, that these two solutions for $x_0$ coincide, let us multiply
  Equation~\eqref{eq:18} with $(q_0+q_3)$ from the right hand side and with
  $\reverse{q_2}$ from the left hand side and use Equation~\eqref{eq:16} to obtain
  \begin{align*}
    \lambda_1 \reverse{q_2}U_1(q_0+q_3)&= \lambda_4 \reverse{q_2}U_4(q_0+q_3),\\
    -\lambda_1\reverse{q_2}q_1r\reverse{(q_0+q_3)}q_1 &= 2\lambda_4 \reverse{q_2}(q_0-q_3)r\reverse{q_2}q_1,\\
    \reverse{q_2}q_1(-\lambda_1\SC(q_1,(q_0+q_3)\reverse{r})) &= \reverse{q_2}q_1(2\lambda_4\SC(\reverse{q_2},\reverse{r}\reverse{(q_0-q_3)})).
  \end{align*}
  Further, it holds
  $\SC(\reverse{q_2},\reverse{r}\reverse{(q_0-q_3)}))=\SC(q_2,(q_0-q_3)r)=\SC(q_2\reverse{r},q_0-q_3)$.
  This shows, that both solutions for $x_0$ coincide, provided
  $\reverse{q_2}q_1\neq 0$. If this expression would be zero, we would have
  $q_1\reverse{q_2}=\reverse{q_1}q_2=0$, which implies that $q_1$ and $q_2$ and
  in turn $U_1$ and $U_2$ are linearly dependent. This is a contradiction to our
  assumptions for Case~3.1.

  To finish the discussion of Case~3.1 we have to show that the two solutions
  for $x_\infty$ are the same. Multiplying Equation~\eqref{eq:18} with
  $\reverse{q_1}$ from the left, and with $(q_0-q_3)$ from the right yields
  \begin{equation*}
    \reverse{q_1}q_2(2\lambda_3\SC(q_0+q_3,q_1\reverse{r}))=\reverse{q_1}q_2(-\lambda_2\SC(q_2,(q_0-q_3)\reverse{r})),
  \end{equation*}
  which shows, that the solutions for $x_\infty$ coincide, provided
  $\reverse{q_1}q_2\neq 0$.

  \paragraph{Case 3.2}
  If, however, $\reverse{q_2}q_1=0$, it follows from Equation~\eqref{eq:16} that
  also $\reverse{(q_0+q_3)}(q_0-q_3)=0$ which shows that $[q_1]=[q_2]$ and
  $[q_0+q_3]=[q_0-q_3]$. But this in turn implies, that $U_1$ and $U_2$ as well
  as $U_3$ and $U_4$ are linearly dependent. Thus we have two degrees of freedom
  for choosing $\lambda_1$, $\lambda_2$, $\lambda_3$, and $\lambda_4$ such that
  \eqref{eq:18} holds, and two linear conditions, that need to be fulfilled,
  i.e. the solutions for $x_0$ and $x_\infty$ should coincide. Therefore, it is
  possible, to choose $\lambda_1$, $\lambda_2$, $\lambda_3$, and $\lambda_4$,
  such that our system of equations is fulfilled.

  \paragraph{Case 3.3}
  In Case~3.2 we have actually treated the situation where $U_1$ and $U_2$ are
  linearly dependent. There, we also showed that this is the case if and only if
  $U_3$ and $U_4$ are linearly dependent. Hence, let us assume here that $U_1$
  and $U_3$ are linearly dependent. We choose $X=\lambda_2 U_2 + \lambda_4 U_4$
  for some yet to be determined $\lambda_2$ and $\lambda_4$. This means, by the
  arguments used in Case~3.1, that Equations~\eqref{eq:sys1} and \eqref{eq:sys3}
  are fulfilled with $x_0 = -\lambda_4 S(q_0-q_3, q_2\reverse{r})$ and $x_\infty
  = \lambda_2 S(q_2, (q_0-q_3)\reverse{r})/2$. The map $[\lambda_2,\lambda_4]
  \mapsto Xq_1$ is a projective map between two projective lines. Hence we can
  select $\lambda_2$ and $\lambda_4$ such that $Xq_1$ is a scalar multiple of
  $q_1$ and $X(q_0+q_3)$ is a scalar multiple of $q_0+q_3$ (and hence also of
  $q_1$).

  Now, by left-multiplying Equation~\eqref{eq:sys1} with $X$, we obtain
  \begin{equation}
    \label{eq:19}
    x_0X(q_0-q_3) - X^2q_2 = 0.
  \end{equation}
  Note that $X^2$ is a scalar. We plug \eqref{eq:19} into
  Equation~\eqref{eq:sys3} to get $X^2 = -2x_0x_\infty$. Left-multiplying
  Equation~\eqref{eq:sys2} with $X$ we can transform it into
  Equation~\eqref{eq:sys4}. Thus, Equations~\eqref{eq:sys2} and \eqref{eq:sys4}
  actually just give \emph{one further} linear condition on $\lambda_2$ and
  $\lambda_4$ so that a solution does exist.

  \paragraph{Case 4:} Here, we discuss the cases where it least one $q_1$,
  $q_2$, $q_0 + q_3$ or $q_0 - q_3$ is zero.

  \paragraph{Case 4.1:}
  Let us assume $q_1=0$. Then \eqref{eq:sys2} gives $x_\infty =0$. Using this,
  we see from \eqref{eq:sys3}, that we should choose $X=\lambda_4 U_4$. From
  this, we get $x_0=-\lambda_4\SC(q_0-q_3,q_2\reverse{r})$ by
  Equation~\eqref{eq:sys1}. Equation~\eqref{eq:16} implies, that \eqref{eq:sys4}
  is fulfilled, hence we found a solution. The other cases, where only one of
  the quaternions $q_2$, $q_0 + q_3$, $q_0 - q_3$ is zero, are similar.

  \paragraph{Case 4.2:}
  Let us assume $q_1 = q_0 + q_3 = 0$, i.e., $U_1 = U_3 = 0$. In this case,
  Eqs.~\eqref{eq:sys2} and \eqref{eq:sys4} are already fulfilled. Thus, after
  choosing $X\coloneqq \lambda_2U_2+\lambda_4U_4$, we are left with one equation
  to determine $x_0=-\lambda_4\SC(q_0-q_3,q_2\reverse{r})$ and one to determine
  $x_\infty=\lambda_2\SC(q_2,(q_0-q_3)\reverse{r})/2$.

  \paragraph{Case 4.3:}
  The case $q_2 = q_0 - q_3 = 0$, i.e., $U_2 = U_4 = 0$ is similar to Case~3.2.

  \paragraph{Case 4.4:}
  Let us assume $q_1 = q_0 - q_3 = 0$, i.e., $U_1=U_4=0$. In this case,
  Eqs.~\eqref{eq:sys2} and \eqref{eq:sys3} immediately yield $x_\infty=0$. We
  are only left with two equations $Xq_2=0$ and $X(q_0+q_3)=0$ which in general
  will only be fulfilled for $X=0$. However, $x_0$ can be chosen arbitrarily and
  we have found a non-zero solution.

  \paragraph{Case 4.5:}
  Let us assume $q_2=q_0+q_3=0$. In this case, we have $x_0=0$, $X=0$ and
  $x_\infty$ can be chosen arbitrarily.
\end{proof}

\end{document}